\documentclass[12pt]{amsart}

\usepackage[utf8]{inputenc}
\usepackage[T1]{fontenc}

\usepackage{verbatim} 
\usepackage[matrix, arrow, cmtip]{xy}
\usepackage{tikz}
\usetikzlibrary{matrix, positioning, arrows.meta}

\usepackage{amssymb, enumerate, enumitem, longtable, float, graphicx, array}

\makeatletter
\@namedef{subjclassname@2010}{%
  \textup{2010} Mathematics Subject Classification}
\makeatother

\newtheorem{theorem}{Theorem}[section]

\newtheorem{lemma}[theorem]{Lemma}

\theoremstyle{definition}
\newtheorem{definition}[theorem]{Definition}
\newtheorem{numbered_definition}[theorem]{Definition}

\newtheorem{remark}[theorem]{Remark}

\newcommand{\df}[1]{{{\bf #1}}}

\DeclareMathOperator{\im}{im}
\DeclareMathOperator{\st}{st}
\DeclareMathOperator{\mesh}{mesh}
\DeclareMathOperator{\id}{id}

\newcommand{\ptsize}{{1pt}}

\begin{document}
\baselineskip=17pt

\bibliographystyle{abbrv}
\title[Combinatorics of Markov compacta]{Detecting topological properties of Markov compacta with combinatorial properties of their diagrams}

\author{G.~C.~Bell}
\address[G.~C.~Bell]{Department of Mathematics and Statistics, University of North Carolina at Greensboro, Greensboro, NC 27412, USA}
\email{gcbell@uncg.edu}

\author[A.~Nagórko]{A.~Nagórko}
\address[A.~Nagórko]{Faculty of Mathematics, Informatics, and Mechanics, University of Warsaw, Banacha 2, 02-097 Warszawa, Poland}
\email{amn@mimuw.edu.pl}

\thanks{The first author was supported by a UNCG Faculty First Grant. 
This research was supported by the NCN (Narodowe Centrum Nauki) grant no. 2011/01/D/ST1/04144.}

\date{}



\begin{abstract}
  We develop a formalism that allows us to describe Markov compacta with finite sets of diagrams that are building blocks of the entire sequence. This encodes complex, continuous spaces with discrete collections of combinatorial objects. We show that topological properties of the limit (such as $k$-connectedness, local $k$-connectedness or the disjoint arcs property) may be detected by looking at combinatorial properties of the diagrams.

  Markov compacta were introduced by M. Gromov and were motivated by some examples in geometric group theory.
  In particular, boundaries at infinity of hyperbolic groups belong to this
class.
\end{abstract}
\maketitle

\section{Introduction}

Consider the following inverse sequence.

\tikzset{outer/.style={->, >=latex, shorten >= 5pt, shorten <= 5pt, dash pattern=on 1pt off 1pt on 1pt off 1pt on 1pt off 1pt on 1pt off 1pt on 100pt}}
\vspace{2mm}
\begin{center}
\begin{tikzpicture}[scale=0.8]
  \coordinate (A) at (0,0); \coordinate (B) at (3, 1.5);
  \coordinate (C) at (3,-1.5); \coordinate (D) at (6, 2);
  \coordinate (E) at (6,1); \coordinate (F) at (6, -1);
  \coordinate (G) at (6,-2); \coordinate (H) at (9, 2.4);
  \coordinate (I) at (9, 1.6); \coordinate (J) at (9, 1.4);
  \coordinate (K) at (9, 0.6); \coordinate (L) at (9, -0.6);
  \coordinate (M) at (9, -1.4); \coordinate (N) at (9, -1.6);
  \coordinate (O) at (9, -2.4);
  
  \filldraw (A) circle (\ptsize) (B) circle (\ptsize) (C) circle (\ptsize)
    (D) circle (\ptsize) (E) circle (\ptsize) (F) circle (\ptsize)
    (G) circle (\ptsize);

  \draw[->, >=latex, shorten >= 5pt, shorten <= 5pt] (B) -- (A);
  \draw[->, >=latex, shorten >= 5pt, shorten <= 5pt] (C) -- (A);
  \draw[->, >=latex, shorten >= 5pt, shorten <= 5pt] (D) -- (B);
  \draw[->, >=latex, shorten >= 5pt, shorten <= 5pt] (E) -- (B);
  \draw[->, >=latex, shorten >= 5pt, shorten <= 5pt] (F) -- (C);
  \draw[->, >=latex, shorten >= 5pt, shorten <= 5pt] (G) -- (C);

  \draw[outer] (H)--(D); \draw[outer] (I)--(D);
  \draw[outer] (J)--(E); \draw[outer] (K)--(E);
  \draw[outer] (L)--(F); \draw[outer] (M)--(F);
  \draw[outer] (N)--(G); \draw[outer] (O)--(G);
  
  \node at (0,-3) { $K_1$ };
  \draw[->, shorten >= 10pt, shorten <= 10pt] (3,-3) -- node[above] {$p_1$} (0,-3);
  \node at (3,-3) { $K_2$ };
  \draw[->, shorten >= 10pt, shorten <= 10pt] (6,-3) -- node[above] {$p_2$} (3,-3);
  \node at (6,-3) { $K_3$ };
  \draw[->, shorten >= 10pt, shorten <= 10pt] (9,-3) -- node[above] {$p_3$} (6,-3);
  \node at (9,-3) { $\cdots$ };
\end{tikzpicture}
\end{center}
\vspace{1mm}

The $n$-th space $K_n$ of this sequence is a discrete space with $2^{n-1}$ elements and each bonding map $p_n$ has $2$-element fibers.
The inverse limit of this sequence is homeomorphic to the Cantor set.
Observe that the whole sequence can be recreated using a single diagram

\begin{center}
\begin{tikzpicture}[scale=0.8]
  \filldraw 	(0,0) circle (\ptsize)
  		(3, 1) circle (\ptsize)
		(3, -1) circle (\ptsize);
  \node (base) at (0,0) { };
  \draw[->, >=latex, shorten >= 1pt, shorten <= 5pt] (3, 1) -- (base);
  \draw[->, >=latex, shorten >= 1pt, shorten <= 5pt] (3, -1) -- (base);
\end{tikzpicture}
\end{center}
which can be considered a ``building block" of the sequence. Similarily, the classical inverse sequence used to define a solenoid

\vspace{1mm}
\begin{center}
\begin{tikzpicture}[scale=0.7]
\coordinate (A) at (0,0);
\coordinate (B) at (2,2);
\coordinate (C) at (2,-2);
\filldraw (A) circle(\ptsize);
\filldraw (B) circle(\ptsize);
\filldraw (C) circle(\ptsize);
\draw (B)--(A)--(C);
\draw [densely dotted] (B)--(C);
\coordinate (D) at (4,0);
\coordinate (E) at (6,2);
\coordinate (F) at (6,-2);
\coordinate (G) at (4.5,0);
\coordinate (H) at (6.5,2);
\coordinate (I) at (6.5,-2);
\draw [densely dotted] (H)--(F);
\draw [white, line width=3pt] (I)--(E);
\draw (E)--(I);
\draw [white, line width=3pt] (H)--(G);
\draw [white, line width=3pt] (I)--(G);

\filldraw (D) circle(\ptsize);
\filldraw (E) circle(\ptsize);
\filldraw (F) circle(\ptsize);
\filldraw (G) circle(\ptsize);
\filldraw (H) circle(\ptsize);
\filldraw (I) circle(\ptsize);

\draw (I)--(G)--(H);
\draw (E)--(D)--(F);

\coordinate (D') at (8,0);
\coordinate (E') at (10,2);
\coordinate (F') at (10,-2);
\coordinate (G') at (8.5,0);
\coordinate (H') at (10.5,2);
\coordinate (I') at (10.5,-2);

\coordinate (D'') at (8.15,0);
\coordinate (E'') at (10.15,2);
\coordinate (F'') at (10.15,-2);
\coordinate (G'') at (8.65,0);
\coordinate (H'') at (10.65,2);
\coordinate (I'') at (10.65,-2);

\draw [densely dotted] (F')--(H''); 
\draw [white, line width=2pt] (F'')--(H');
\draw (F'')--(H');
\draw [white, line width=2pt] (I')--(E');
\draw [white, line width=2pt] (I'')--(E'');
\draw (I')--(E');
\draw (I'')--(E'');
\draw [white, line width=2pt] (I')--(G');
\draw [white, line width=2pt] (I'')--(G'');
\draw (I')--(G');
\draw (I'')--(G'');
\draw [white, line width=2pt] (H')--(G');
\draw [white, line width=2pt] (H'')--(G'');
\draw (H')--(G');
\draw (H'')--(G'');
\draw [white, line width=2pt] (F'')--(D'');
\draw [white, line width=2pt] (D'')--(E'');
\draw (F')--(D')--(E');
\draw (F'')--(D'')--(E'');

\filldraw (D') circle(\ptsize);
\filldraw (E') circle(\ptsize);
\filldraw (F') circle(\ptsize);
\filldraw (G') circle(\ptsize);
\filldraw (H') circle(\ptsize);
\filldraw (I') circle(\ptsize);

\filldraw (D'') circle(\ptsize);
\filldraw (E'') circle(\ptsize);
\filldraw (F'') circle(\ptsize);
\filldraw (G'') circle(\ptsize);
\filldraw (H'') circle(\ptsize);
\filldraw (I'') circle(\ptsize);

\draw [->,shorten >= 10pt, shorten <= 10pt] (D)--(2,0);
\draw [->,shorten >= 10pt, shorten <= 10pt] (D')--(6.25,0);
\draw [->,shorten >= 10pt, shorten <= 10pt] (12.5,0)--(10.5,0);
\node at (12.7,0) {$\cdots$};

\end{tikzpicture}
\end{center}
can be encoded using the following two ``building blocks'' and colors
 (represented by solid and dotted lines) that describe the ``gluing'' process.

\vspace{3mm}
\begin{center}
\begin{tikzpicture}
\draw [densely dotted] (3,0)--(3,1.5);
\draw [densely dotted] (4,0)--(4.5,1.5);
\draw [->,shorten >= 3pt, shorten <= 3pt] (4,0.75)--(3,0.75);
\draw [->,shorten >= 3pt, shorten <= 3pt] (1,0.75)--(0,0.75);
\draw [line width=3pt,white] (4,1.5)--(4.5,0);
\draw (4,1.5)--(4.5,0);

\draw (0,0)--(0,1.5);
\draw (1,0)--(1,1.5);
\draw (1.5,0)--(1.5,1.5);
\foreach \t in {0,1,1.5, 3, 4, 4.5}{
\filldraw (\t,0) circle (\ptsize)
             (\t,1.5) circle (\ptsize);
}
\end{tikzpicture}
\end{center}

  \vspace{2mm}
  Both the Cantor set and the solenoid are examples of Markov compacta.
  The class of Markov compacta was introduced by M.~Gromov and was motivated by some examples in geometric group theory.
  In particular, boundaries at infinity of hyperbolic groups belong to this
class~\cite{pawlik2015}.

  In the present paper we develop a formalism that allows us to encode Markov compacta using finite sets of diagrams and gluing rules similar to the diagrams depicted above.
  This allows us to encode complex, continuous topological spaces using a discrete collection of finite complexes.
  We show that some topological properties of the limit space (such as connectedness, local connectedness or the disjoint arcs property) can be detected by the combinatorial properties of the diagrams.
  In particular, we give examples of Markov spaces that are homeomorphic to Menger cubes and N\"obeling spaces.
  Detecting topological properties that characterize these spaces is important for applications.
  For example, by theorems of Hensel-Przytycki~\cite{henselprzytycki2011} and Gabai~\cite{gabai2011} the boundary at infinity of a curve complex (which is a hyperbolic graph) of a sphere punctured $n+5$ times is homeomorphic to the $n$-dimensional N\"obeling space. Moreover, Dymara and Osajda~\cite{DymaraOsajda} have shown that the boundary of a right-angled hyperbolic building is a universal Menger space.
  The characterization of boundaries of curve complexes of higher genus surfaces with punctures is open.

  We focus on the one-dimensional case, where all polyhedra in the sequence are graphs. This case covers a wide class of interesting examples.
    

  This paper opens a new avenue of research: to determine 
  which topological properties of Markov compacta can be detected by combinatorial properties of its diagram. 
  
\section{Markov Spaces}

\begin{definition}
  Let $K$ be a simplicial complex.
  Let $\tau(K)$ be the triangulation of~$K$.
    We let $\beta K$ denote the barycentric subdivision of $K$ (i.e., the same metric space but with finer triangulation $\tau(\beta K)$).
  A \df{coloring} of $K$ is a map from $\tau(K)$ to $\mathbb{N}$.
  A \df{colored complex} is a simplicial complex with fixed coloring.
  We let $c(K) \colon K \to \mathbb{N}$ denote the coloring of a colored complex $K$.
  A \df{colored embedding} is a simplicial embedding $i \colon K \to L$ of colored complexes $K$ and $L$ such that  $c(K) = c(L) \circ i$.
\end{definition}

\begin{definition}
  Let $K$ and $L$ be simplicial complexes and let $p \colon K \to L$.
  We say that $p$ is \df{quasi-simplicial} if it is a simplicial map into $\beta L$.
\end{definition}

\begin{definition}
  A \df{production $P$} is a simplicial or quasi-simplicial map 
  $P \colon K \to L$ from a colored complex $P$ into a colored complex $K$.
  We call $K$ the \df{top} complex of the production and denote it 
  by $t(P)$.
  We call $L$ the \df{bottom} complex of the production and denote it
  by $b(P)$.
\end{definition}

\begin{definition}
  Let $S$ and $T$ be productions.
  A \df{gluing $G$} is a pair of maps $(G_t \colon t(S) \to t(T), 
    G_b \colon b(S) \to b(T))$ that are colored embeddings 
    such that the following diagram is commutative
  \[
    \xymatrix@M=8pt@C=35pt{
    t(S) \ar[r]^{G_t} \ar[d]^{S} & t(T) \ar[d]^{T} \\
    b(S) \ar[r]^{G_b}  & b(T) \\
    }
  \]  
\end{definition}

\begin{definition}
  A \df{Markov diagram $\mathcal{D}$} is a triple $\mathcal{D} = (S, \mathcal{P}, \mathcal{G})$, where
  $S$ is a simplicial complex that we call \df{the starting complex of 
  $\mathcal{D}$}, $\mathcal{P}$ is a set of productions that we call \df{the production set of $\mathcal{D}$} and $\mathcal{G}$ is a set of gluings that we call \df{the set of gluing rules of $\mathcal{D}$}.
\end{definition}

\begin{definition}
  Let $\mathcal{D} = (S, \mathcal{P}, \mathcal{G})$ be a Markov diagram.
  Let $\Gamma = (V, E)$ be a directed graph.
  Let $P \colon V \to \mathcal{P}$ be a labelling of vertices of $\Gamma$ by productions of $\mathcal{D}$.
  Let $G \colon E \to \mathcal{G}$ be a labelling of edges of $\Gamma$ by gluings of $\mathcal{D}$. 
  If for each directed edge $e = (u, v)$ of $\Gamma$ the following diagram is commutative
  \[
    \xymatrix@M=8pt@C=35pt{
    t(P(u)) \ar[r]^{G(e)_t} \ar[d]^{P(u)} & t(P(v)) \ar[d]^{P(v)} \\
    b(P(u)) \ar[r]^{G(e)_b}  & b(P(v)) \\
    }
  \]
  then we say that $\Gamma$ is an \df{assembly graph} for $\mathcal{D}$.
\end{definition}

\begin{definition}
  Let $K, L$ be colored complexes.
  Let $f \colon K \to L$ be a simplicial or a quasi-simplicial map.
  Let $\mathcal{D} = (S, \mathcal{P}, \mathcal{G})$ be a Markov diagram.
  Let $\Gamma = (V, E)$ be an assembly graph for $\mathcal{D}$ with labelling $P \colon V \to \mathcal{P}$ of vertices and labelling $G \colon E \to \mathcal{G}$ of edges.

  Let $\mathcal{T} = \{ f_v \colon t(P(v)) \to K \}_{v \in V}$
  and $\mathcal{B} = \{ g_v \colon b(P(v)) \to L \}_{v \in V}$ be collections of colored embeddings.
  Assume that for each edge $e = (u, v) \in V$ the following diagram is commutative
  \[
    \xymatrix@M=8pt@C=35pt{
    K \ar[d]^f & \ar[l]^{f_u} t(P(u)) \ar[r]^{G(e)_t} \ar[d]^{P(u)} & t(P(v)) \ar[d]^{P(v)} \ar[r]^{f_v} & K \ar[d]^f \\
    L & \ar[l]^{g_u} b(P(u)) \ar[r]^{G(e)_b}  & b(P(v)) \ar[r]^{g_v} & L \\
    }
  \]
  Assume that $\{ \im f_v \}_{v \in V}$ is a cover of $K$ that is closed under intersection and $\{ \im g_v \}_{v \in V}$ is a cover of $L$ that is closed under intersection.
  
  Then we say that $(\mathcal{T}, \mathcal{B})$ is a \df{chart} for $f$.
  We call $\mathcal{T}$ the \df{top chart} and $\mathcal{B}$ the \df{bottom chart} for $f$.
  The structure $(\Gamma, \mathcal{T}, \mathcal{B})$ satisfying the above conditions is called a \df{decomposition of $f$} with respect to $\mathcal{D}$.
\end{definition}

\begin{definition}
  Let $\mathcal{D}$ be a Markov diagram. A \df{Markov sequence} is an inverse sequence
  \[
   K_1\xleftarrow{p_1}K_2\xleftarrow{p_2}K_3\xleftarrow{p_3}\cdots
  \]
  along with a decompositions of all bonding maps $p_i$ over $\mathcal{D}$.
  The limit of the sequence
  \[
   \lim_{\longleftarrow} \left( K_1\xleftarrow{p_1}K_2\xleftarrow{p_2}K_3\xleftarrow{p_3}\cdots \right)
  \]
  is called a \df{Markov space}. If complexes $K_i$ are $1$-dimensional, then we call the limit a \df{Markov curve}.
\end{definition}


\subsection{The $1$--$8$ sequence}\label{ss:18}

All graphs in this section are mono-colored. The $1$--$8$ sequence is named after a single edge production $P_8$ it contains.

\vspace{1mm}
\begin{center}
\begin{tikzpicture}
\node at (-0.5,1) {$P_8$:};
\coordinate (X) at (2,0); \coordinate (Y) at (0,0);
\draw [black] (X)--(Y); \fill (X) circle(\ptsize); \fill (Y) circle(\ptsize);
\coordinate (A) at (2,2); \coordinate (B) at (2,1);
\coordinate (F) at (1,2); \coordinate (C) at (1,1);
\coordinate (E) at (0,2); \coordinate (D) at (0,1);
\fill (A) circle(\ptsize); \fill (B) circle(\ptsize);
\fill (C) circle(\ptsize); \fill (D) circle(\ptsize);
\fill (E) circle(\ptsize); \fill (F) circle(\ptsize);
\draw [black] (A)--(B)--(C)--(D)--(E)--(F)--(A);
\draw [black] (F)--(C);
\draw [->,shorten >= 5pt, shorten <= 5pt] (C)--(1,0);
\end{tikzpicture}
\end{center}

The production map $P_8$ is the vertical quasi-simplicial 
  projection from the top $8$-shaped graph $t(P_8)$ onto the single-edge 
  graph $b(P_8)$ on the bottom.
The production set is completed by a single vertex production $P_1$.

\begin{center}
\begin{tikzpicture}
\node at (-0.5,0) {$P_1$:};
\coordinate (Y) at (0,0);
\fill (Y) circle(\ptsize);
\coordinate (E) at (2,0); \coordinate (D) at (3,0);
\fill (D) circle(\ptsize);
\fill (E) circle(\ptsize);
\draw [black] (D)--(E);
\draw [->,shorten >= 10pt, shorten <= 10pt] (E)--(Y);
\end{tikzpicture}
\end{center}

The production map $P_1$ projects horizontally the right edge graph $t(P_1)$ onto a single vertex graph $b(P_1)$ on the left.
We let $\mathcal{P} = \{ P_1, P_8 \}$ be the production set of the $1$--$8$ sequence.
There are two gluing rules in the $1$--$8$ sequence, the left gluing $G_l$ and the right gluing $G_r$.

\vspace{1mm}
\begin{center}
\begin{tikzpicture}
\coordinate (X) at (2,0); \coordinate (Y) at (0,0);
\draw [black] (X)--(Y); \fill (X) circle(\ptsize); \fill (Y) circle(\ptsize);
\coordinate (A) at (2,2); \coordinate (B) at (2,1);
\coordinate (F) at (1,2); \coordinate (C) at (1,1);
\coordinate (E) at (0,2); \coordinate (D) at (0,1);
\fill (A) circle(\ptsize); \fill (B) circle(\ptsize);
\fill (C) circle(\ptsize); \fill (D) circle(\ptsize);
\fill (E) circle(\ptsize); \fill (F) circle(\ptsize);
\draw [black] (A)--(B)--(C)--(D)--(E)--(F)--(A);
\draw [black] (F)--(C);
\draw [->,shorten >= 5pt, shorten <= 5pt] (C)--(1,0);
\coordinate (G) at (-2, 0);
\coordinate (H) at (-2, 1);
\coordinate (I) at (-2, 2);
\draw [black] (I)--(H);
\fill (G) circle(\ptsize); \fill (H) circle(\ptsize); 
\fill (I) circle(\ptsize);
\draw [->, shorten >= 5pt, shorten <= 5pt] (H) -- node [above] { $G_l$ }(D) ;

\coordinate (G') at (4, 0);
\coordinate (H') at (4, 1);
\coordinate (I') at (4, 2);
\draw [black] (I')--(H');
\fill (G') circle(\ptsize); \fill (H') circle(\ptsize); 
\fill (I') circle(\ptsize);
\draw [->, shorten >= 5pt, shorten <= 5pt] (H') -- node [above] { $G_r$ }(B) ;

\end{tikzpicture}
\end{center}

We let $\mathcal{G} = \{ G_l, G_r \}$ be the set of gluing rules of the $1$--$8$ sequence.
The starting graph is a single edge graph 

\begin{center}
\begin{tikzpicture}
\node at (-0.5,0) {$S$:};
\filldraw (0,0) circle (1pt); \filldraw (1,0) circle (1pt); \draw [black] (0,0) -- (1,0); 
\end{tikzpicture}
\end{center}

The triple $\mathcal{D} = (S, \mathcal{P}, \mathcal{G})$ is a Markov diagram, which we call the $1$--$8$ diagram. Below are first six elements of a Markov sequence that corresponds to this diagram.

\newcommand{\w}{22mm}
\begin{center}
\begin{longtable}{m{\w}m{2mm}m{\w}m{2mm}m{\w}m{2mm}m{\w}m{2mm}m{2cm}}
	\includegraphics[width=\w]{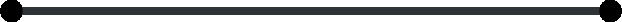} &
	$\leftarrow$ &
	\includegraphics[width=\w]{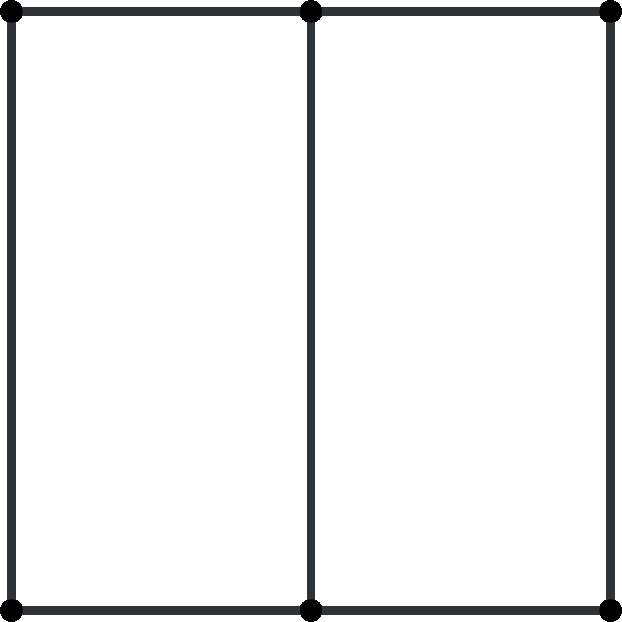} &
	$\leftarrow$ &
	\includegraphics[width=\w]{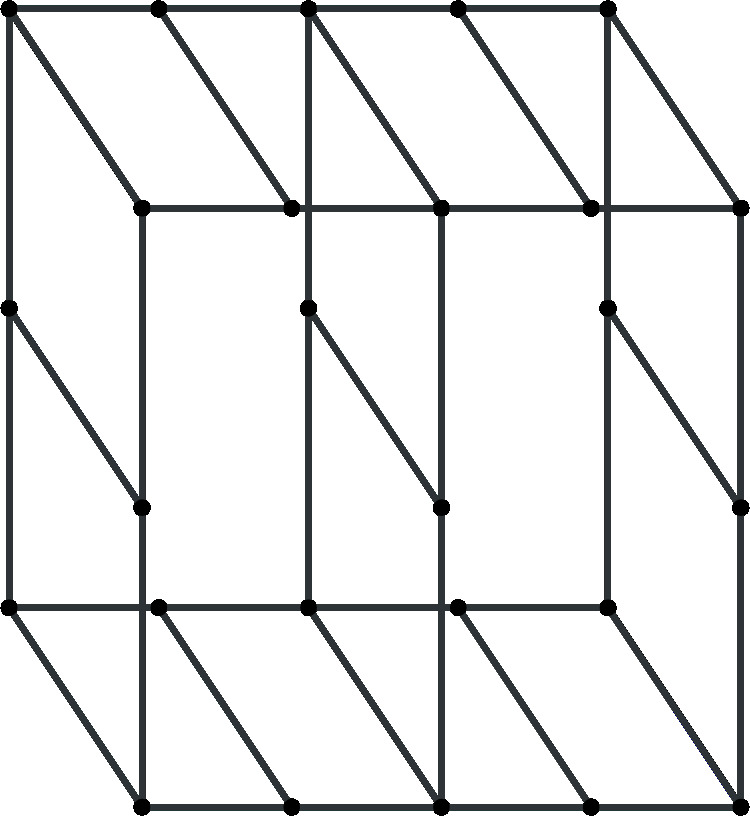} &
	$\leftarrow$ &
    &
    &
    \\
	& & & & & & & &
    \\
    &
	$\leftarrow$ &
	\includegraphics[width=\w]{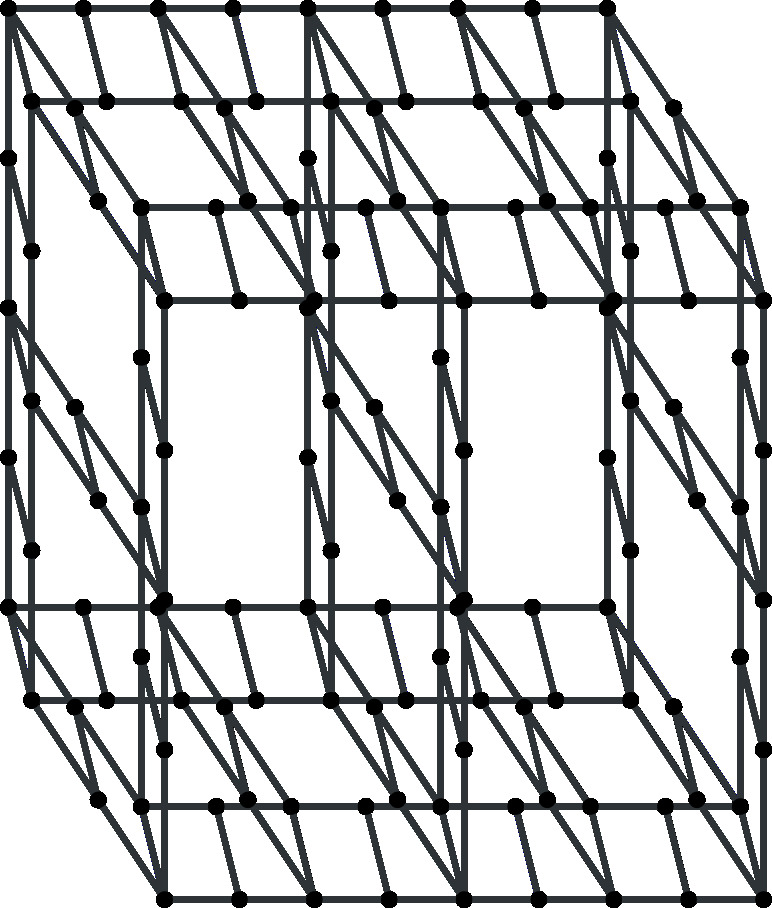} &
	$\leftarrow$ &
	\includegraphics[width=\w]{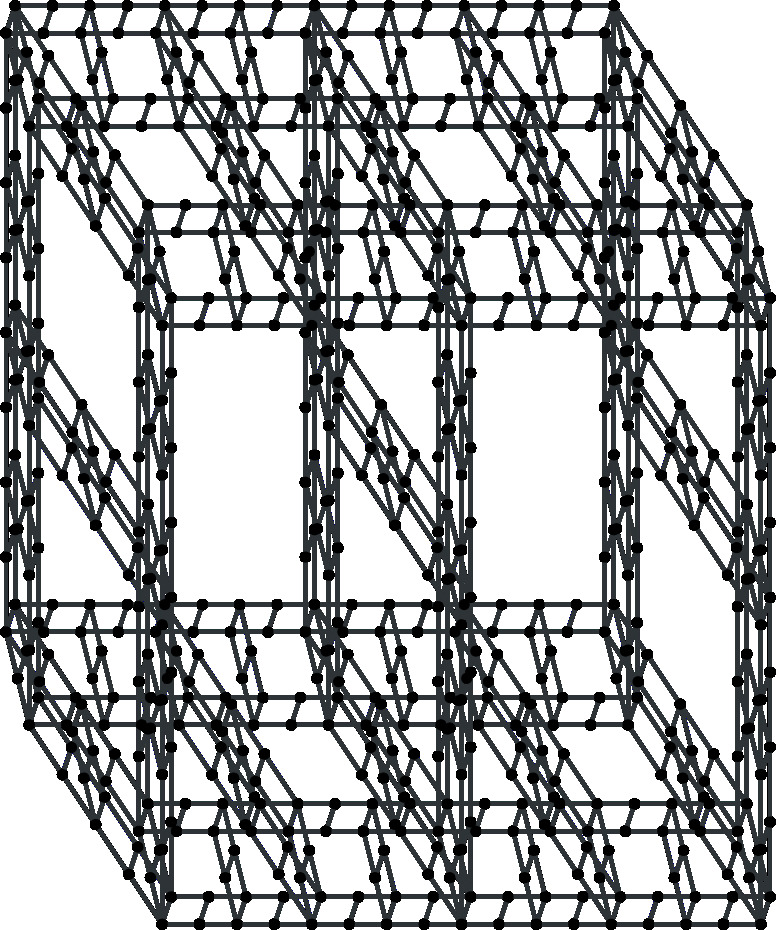} &
	$\leftarrow$ &
	\includegraphics[width=\w]{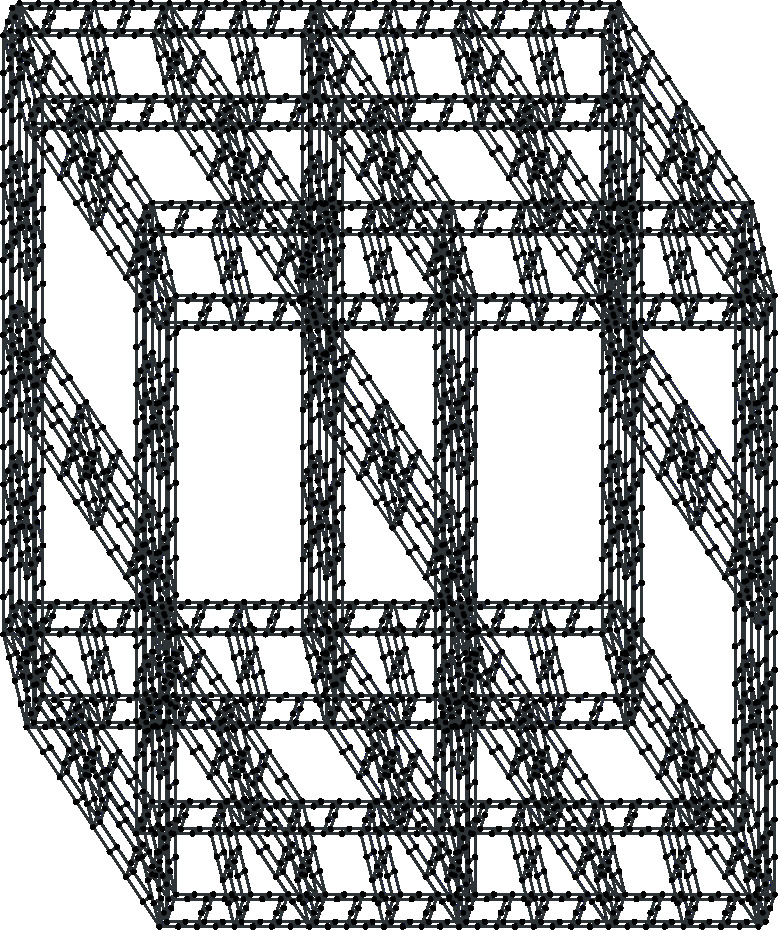} &
	$\leftarrow$
	& $\cdots$
\end{longtable}
\end{center}

We will show how the first bonding map $p_1$ decomposes over $\mathcal{D}$.
We take an assembly graph $\Gamma_1$:

\begin{center}
\begin{tikzpicture}
\coordinate (A) at (0,0); \coordinate (B) at (2,0);
\coordinate (C) at (4,0);
\fill (A) circle(\ptsize) node [above] { $P_1$ } node [below] { \footnotesize $\{v_1\}$ }; 
\fill (B) circle(\ptsize) node [above] { $P_8$ } node [below] { \footnotesize $\overline{v_1v_2}$ };
\fill (C) circle(\ptsize) node [above] { $P_1$ } node [below] { \footnotesize $\{v_2\}$ };
\draw [->,> = latex, shorten >= 1pt] (A) -- node [above] { $G_l$ }(B);
\draw [->,> = latex, shorten >= 1pt] (C) -- node [above] { $G_r$ }(B);
\end{tikzpicture}
\end{center}

The vertices of $\Gamma_1$ are labeled with productions $P_1$ and $P_8$, edges are labeled with gluings $G_l$ and $G_r$. 
This yields a commutative diagram of upper and lower charts, which is a decomposition of $p_1$ over $\Gamma_1$.


The other bonding maps can be decomposed in a similar way.
We will show in the later sections that the limit of the $1$--$8$ sequence is homeomorphic to the universal Menger curve $M^3_1$.







\begin{definition}
  We say that a Markov sequence is an \df{elementary Markov sequence} if the bottom graphs of its productions are either single vertex graphs or single edge graphs.
\end{definition}


For a simplicial complex, the underlying polyhedron has two topologies, the Whitehead (weak) topology and the metric topology.
The weak topology is metrizable if and only if the complex is locally finite and
it coincides with the metric topology in this case.
Since we work in the metric category with complexes that are not locally finite, {\bf we always assume the metric topology on simplicial complexes}.

\begin{definition}
  Let $K$ be simplicial complex.
  We let $\tau(K)$ denote the \df{triangulation of~$K$} (the set of simplices of $K$).
  We let $V(K)$ denote the \df{vertex set of $K$}.
  We let $\beta K$ denote the barycentric subdivision of $K$ (i.e., the same metric space but with finer triangulation $\tau(\beta K)$).
\end{definition}


\begin{numbered_definition}\label{def:geodesic polyhedron}
  Let $K$ be a simplicial complex. Let $\kappa \in (0, \infty)$.
  We define \df{a geodesic metric of scale $\kappa$ on $K$} in the following way.
  On each $\delta \in \tau(K)$ we take a Euclidean metric with edge length $\kappa$. We extend this on $K$ to the unique geodesic metric.
\end{numbered_definition}

\begin{lemma}\label{lem:metrics}
  The metric defined in Definition~\ref{def:geodesic polyhedron} induces the metric topology.
\end{lemma}




\section{Examples}

\subsection{Cantor set}

A Markov diagram that generates a Markov sequence described in the Introduction that generates the Cantor set contains a single point 
\raisebox{1mm}{\tikz \filldraw (0,0) circle (1pt);} as a starting graph, a single production 
\begin{center}
\begin{tikzpicture}
  \filldraw 	(0,0) circle (1pt)
  		(1, 0.5) circle (1pt)
		(1, -0.5) circle (1pt);
  \node (base) at (0,0) { };
  \draw[->, >=latex, shorten >= 1pt, shorten <= 5pt] (1, 0.5) -- (base);
  \draw[->, >=latex, shorten >= 1pt, shorten <= 5pt] (1, -0.5) -- (base);
\end{tikzpicture}
\end{center}
that maps a two-vertex graph onto a single-vertex graph with no gluing rules.
There exists a unique inverse sequence corresponding to this 
diagram.
\begin{center}
\begin{longtable}{m{5mm}m{5mm}m{5mm}m{5mm}m{5mm}m{5mm}m{5mm}m{5mm}m{5mm}m{5mm}m{5mm}m{2cm}}
	\includegraphics[height=3cm]{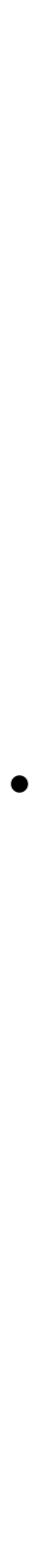} &
	$\leftarrow$ &
	\includegraphics[height=3cm]{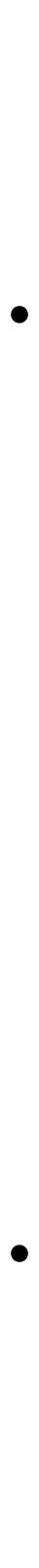} &
	$\leftarrow$ &
	\includegraphics[height=3cm]{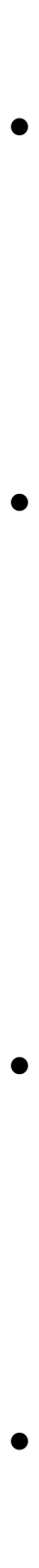} &
	$\leftarrow$ &
	\includegraphics[height=3cm]{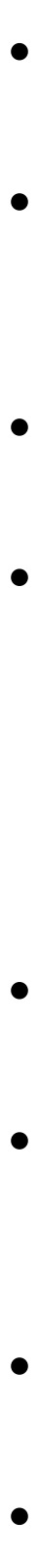} &
	$\leftarrow$ &
	\includegraphics[height=3cm]{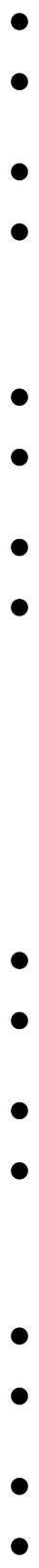} &
	$\leftarrow$ &
	\includegraphics[height=3cm]{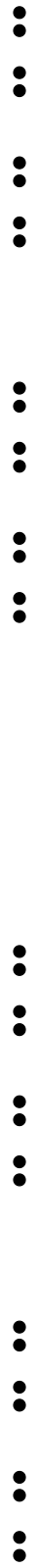} &
	$\leftarrow \ \ \ \cdots$
\end{longtable}
\end{center}

The threads in this inverse limit can be mapped homeomorphically onto the set of binary sequences $\{ 0, 1 \}^\infty$ with the product topology, which is homeomorphic to the Cantor set.


\subsection{The suspension of the Cantor set} If we ``suspend" the diagram for the Cantor set we will get a suspension of the Cantor set in the limit. 
The only non single vertex production is production $P$.

\begin{figure}[H]
\begin{tikzpicture}
  \node at (-0.5, 0) { $P$: };
  \filldraw 	(0,0.5) circle (\ptsize)
  		(0,-0.5) circle (\ptsize)
		(1, 0.5) circle (\ptsize)
		(1, -0.5) circle (\ptsize);

  \draw 	(0,0) circle (\ptsize)
  	   	(0.8, 0) circle (\ptsize)
		(1.2, 0) circle (\ptsize);
		  
  \draw[->, shorten >= 5pt, shorten <= 5pt] (1, 0.5) -- (0, 0.5);
  \draw[->, shorten >= 5pt, shorten <= 5pt] (1, -0.5) -- (0, -0.5);
  \draw[->, shorten >= 5pt, shorten <= 5pt] (0.8, 0) -- (0, 0);
  \draw[shorten >= 3pt, shorten <= 3pt] (1, 0.5) -- (0.8, 0);
  \draw[shorten >= 3pt, shorten <= 3pt] (1, 0.5) -- (1.2, 0);
  \draw[shorten >= 3pt, shorten <= 3pt] (1, -0.5) -- (0.8, 0);
  \draw[shorten >= 3pt, shorten <= 3pt] (1, -0.5) -- (1.2, 0);
  \draw[shorten >= 3pt, shorten <= 3pt] (0, -0.5) -- (0, 0);
  \draw[shorten >= 3pt, shorten <= 3pt] (0, 0.5) -- (0, 0);
\end{tikzpicture}
\end{figure}

Here we use different color to mark end points of the suspension. 
We define a single vertex production 
\begin{center}
\begin{tikzpicture}
\coordinate (Y) at (0,0);
\fill (Y) circle(\ptsize);
\coordinate (E) at (2,0); \coordinate (D) at (3,0);
\fill (E) circle(\ptsize);
\draw [->,shorten >= 10pt, shorten <= 10pt] (E)--(Y);
\end{tikzpicture}
\end{center}
corresponding to the horizontal top and bottom fibers of $P$.
We define gluings that glue these single vertex productions into the corresponding fibers of $P$. We will usually omit these when presenting Markov diagrams.
Note that the Markov diagram does not contain a production for the middle hollow vertex of $P$; it is not needed since gluings are performed only along the top and bottom vertices.

We define the starting space 
\begin{tikzpicture}[rotate=90]
  \filldraw 	(0,0.5) circle (\ptsize)
  		(0,-0.5) circle (\ptsize);
		
  \draw 	(0,0) circle (\ptsize);

  \draw[shorten >= 3pt, shorten <= 3pt] (0, -0.5) -- (0, 0);
  \draw[shorten >= 3pt, shorten <= 3pt] (0, 0.5) -- (0, 0);
		 
\end{tikzpicture}.
The following sequence of spaces is produced.

\begin{figure}[H]
\begin{tikzpicture}
  \filldraw 	(0,0.5) circle (\ptsize)
  		(0,-0.5) circle (\ptsize);
  \draw 	(0,0) circle (\ptsize);
  \draw[shorten >= 3pt, shorten <= 3pt] (0, -0.5) -- (0, 0);
  \draw[shorten >= 3pt, shorten <= 3pt] (0, 0.5) -- (0, 0);

  \draw[->, shorten <= 5pt, shorten >= 5pt] (0.75, 0) -- (0,0);
  
  \filldraw	(1.25, 0.5) circle (\ptsize)
		(1.25, -0.5) circle (\ptsize);
  \draw    	(0.75, 0) circle (\ptsize)
		(1.75, 0) circle (\ptsize);
  \draw[shorten >= 3pt, shorten <= 3pt] (1.25, 0.5) -- (0.75, 0);
  \draw[shorten >= 3pt, shorten <= 3pt] (1.25, 0.5) -- (1.75, 0);
  \draw[shorten >= 3pt, shorten <= 3pt] (1.25, -0.5) -- (0.75, 0);
  \draw[shorten >= 3pt, shorten <= 3pt] (1.25, -0.5) -- (1.75, 0);

  \draw[<-, shorten <= 5pt, shorten >= 5pt] (1.75, 0) -- (2.35,0);

  \filldraw	(3, 0.5) circle (\ptsize)
		(3, -0.5) circle (\ptsize);
  \draw    	(2.35, 0) circle (\ptsize)
  		(2.65, 0) circle (\ptsize)
		(3.35, 0) circle (\ptsize)
		(3.65, 0) circle (\ptsize);
		
  \draw[shorten >= 3pt, shorten <= 3pt] (3, 0.5) -- (2.35, 0);
  \draw[shorten >= 3pt, shorten <= 3pt] (3, 0.5) -- (2.65, 0);
  \draw[shorten >= 3pt, shorten <= 3pt] (3, 0.5) -- (3.35, 0);
  \draw[shorten >= 3pt, shorten <= 3pt] (3, 0.5) -- (3.65, 0);

  \draw[shorten >= 3pt, shorten <= 3pt] (3, -0.5) -- (2.35, 0);
  \draw[shorten >= 3pt, shorten <= 3pt] (3, -0.5) -- (2.65, 0);
  \draw[shorten >= 3pt, shorten <= 3pt] (3, -0.5) -- (3.35, 0);
  \draw[shorten >= 3pt, shorten <= 3pt] (3, -0.5) -- (3.65, 0);

  \draw[<-, shorten <= 5pt, shorten >= 5pt] (3.65, 0) -- (4.35,0);

  \node at (4.5, 0) {$\ldots$};
\end{tikzpicture}
\end{figure}

Observe that the limit space is connected, but not locally connected. 
The horizontal cross sections are homeomorphic to the Cantor set, with the exception of the end points.


\subsection{A diamond sequence}

A diamond sequence is an elementary Markov sequence that starts with a single edge and contains a trivial vertex production 
along with a single edge production.

\begin{center}
\begin{tikzpicture}
  \filldraw 	(0,0.5) circle (1pt)
  		(0,-0.5) circle (1pt)
		(1, 0.5) circle (1pt)
		(1, -0.5) circle (1pt)
  	   	(0.8, 0) circle (1pt)
		(1.2, 0) circle (1pt);
		  
  \draw[->, shorten >= 3pt, shorten <= 5pt] (0.8, 0) -- (0, 0);
  \draw[shorten >= 3pt, shorten <= 3pt] (1, 0.5) -- (0.8, 0);
  \draw[shorten >= 3pt, shorten <= 3pt] (1, 0.5) -- (1.2, 0);
  \draw[shorten >= 3pt, shorten <= 3pt] (1, -0.5) -- (0.8, 0);
  \draw[shorten >= 3pt, shorten <= 3pt] (1, -0.5) -- (1.2, 0);
  \draw[shorten >= 3pt, shorten <= 3pt] (0, -0.5) -- (0, 0.5);
\end{tikzpicture}
\end{center}

There are two gluing rules on fibers of the edge production.

\renewcommand{\w}{17mm}
\begin{figure}[H]
\begin{longtable}{m{5mm}m{2mm}m{\w}m{2mm}m{\w}m{2mm}m{\w}m{2mm}m{\w}m{2mm}m{2cm}}
	\includegraphics[height=\w]{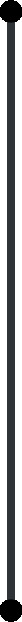} &
	$\leftarrow$ &
	\includegraphics[width=\w]{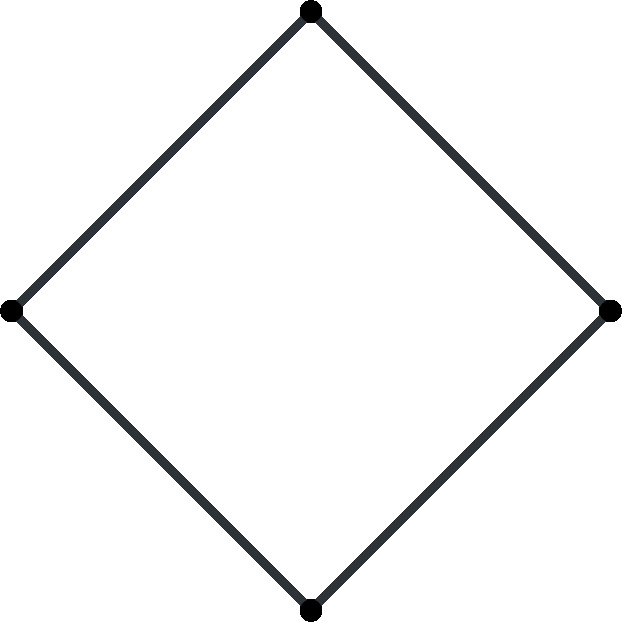} &
	$\leftarrow$ &
	\includegraphics[width=\w]{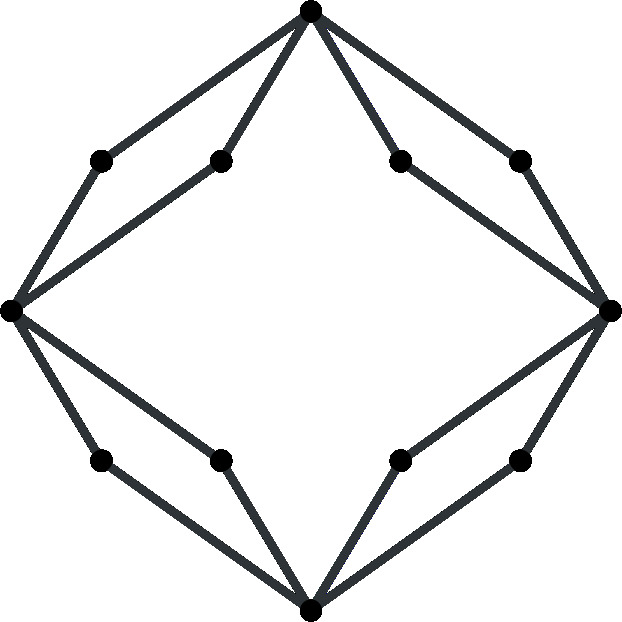} &
	$\leftarrow$ &
	\includegraphics[width=\w]{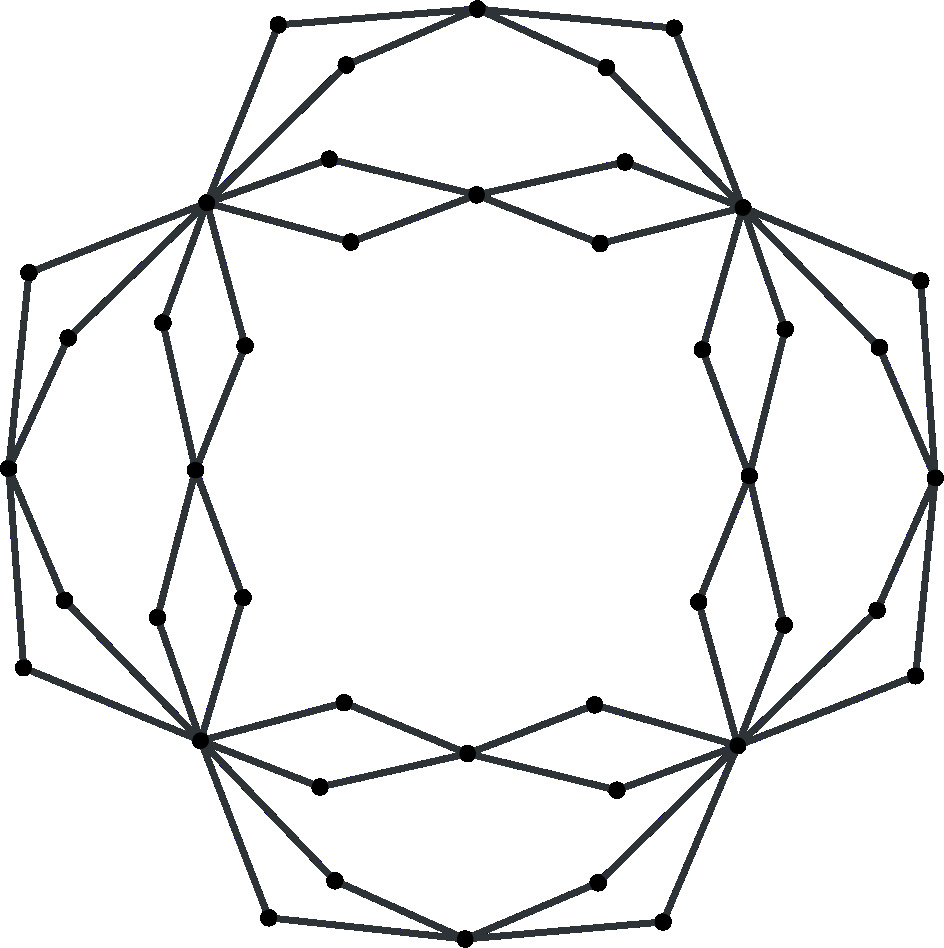} &
    $\leftarrow$ &
	\includegraphics[width=\w]{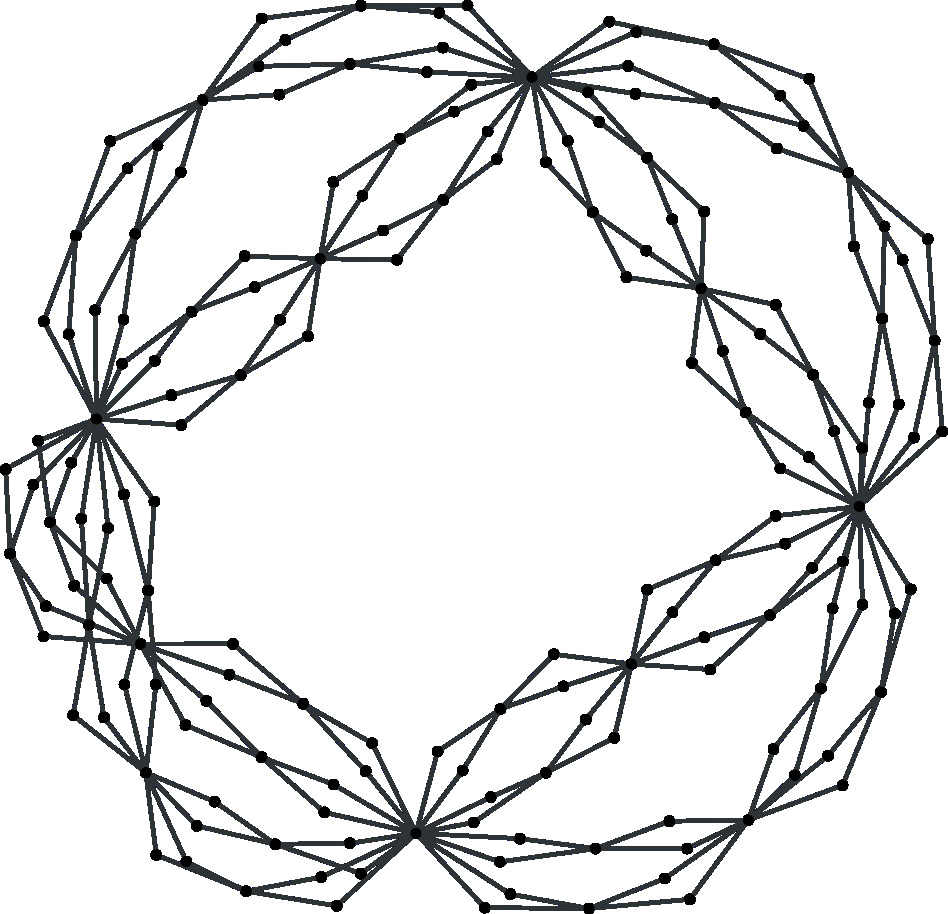} &
    $\leftarrow$ &
    $\cdots$
\end{longtable}
\caption{
  A diamond sequence. The layout of the last two graphs was generated using SFDP spring-block algorithm, hence the irregularities.
}
\end{figure}

The limit of this sequence is a $1$-dimensional connected and locally connected compactum that we call a diamond curve.


\subsection{Join of two Cantor sets}

An elementary Markov sequence that starts with a single edge and contains a double vertex production 
along with a single edge production depicted in the following diagram.

\vspace{1mm}
\begin{center}
\begin{tikzpicture}
  \filldraw 	(0,0.5) circle (1pt)
  		(0,-0.5) circle (1pt)
		(1, 0.5) circle (1pt)
		(1, -0.5) circle (1pt)
  	   	(1.5, 0.5) circle (1pt)
		(1.5, -0.5) circle (1pt);
		  
  \draw[->, shorten >= 3pt, shorten <= 5pt] (0.8, 0) -- (0.2, 0);
  \draw[shorten >= 3pt, shorten <= 3pt] (1, 0.5) -- (1, -0.5);
  \draw[shorten >= 3pt, shorten <= 3pt] (1.5, 0.5) -- (1.5, -0.5);
  \draw[shorten >= 3pt, shorten <= 3pt] (1, 0.5) -- (1.5, -0.5);
  \draw[shorten >= 3pt, shorten <= 3pt] (1.5, 0.5) -- (1, -0.5);
\end{tikzpicture}
\end{center}

The first six graphs from the sequence are shown below. The inverse limit is homeomorphic to a join of two Cantor sets.

\renewcommand{\w}{22mm}
\begin{longtable}{m{\w}m{2mm}m{\w}m{2mm}m{\w}m{2mm}m{\w}m{2mm}m{2cm}}
	\includegraphics[width=\w]{diagrams/menger_0.jpg} &
	$\leftarrow$ &
	\includegraphics[width=\w]{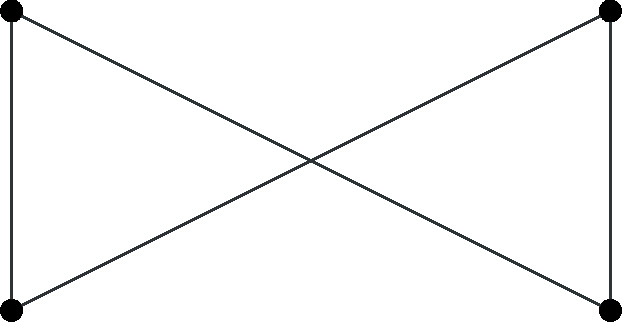} &
	$\leftarrow$ &
	\includegraphics[width=\w]{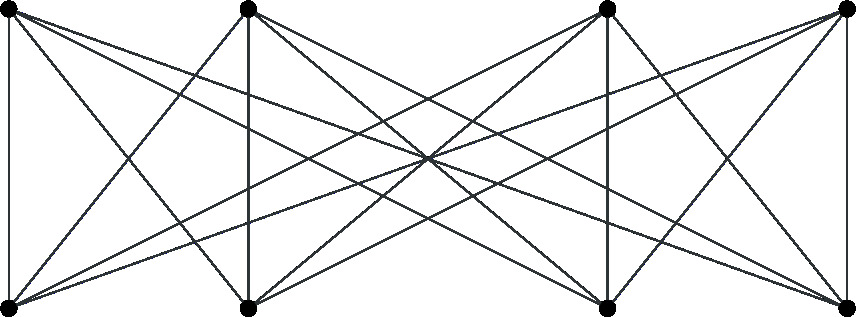} &
	$\leftarrow$ &
    &
    &
    \\
	& & & & & & & &
    \\
    &
	$\leftarrow$ &
	\includegraphics[width=\w]{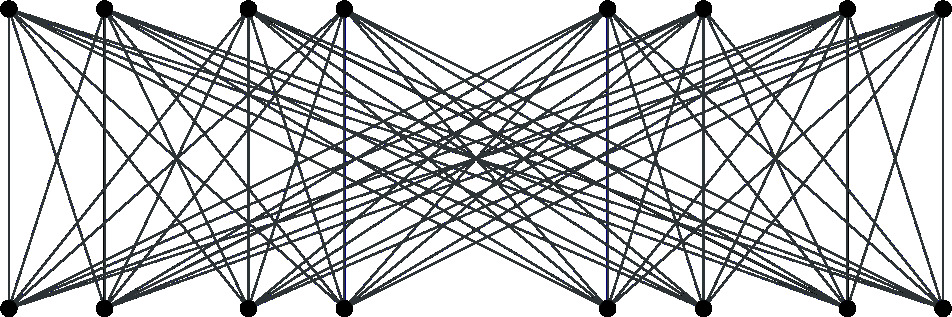} &
	$\leftarrow$ &
	\includegraphics[width=\w]{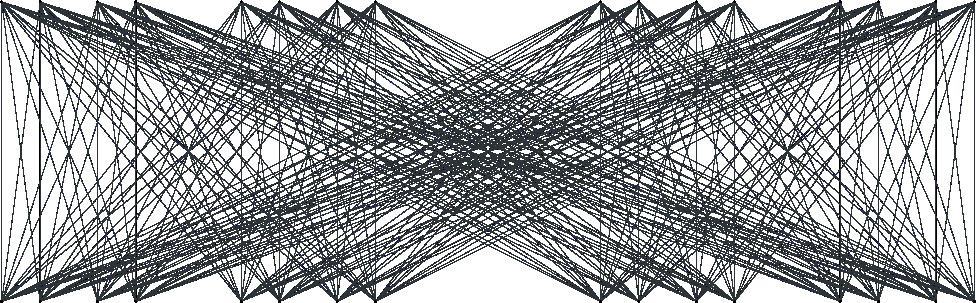} &
	$\leftarrow$ &
	\includegraphics[width=\w]{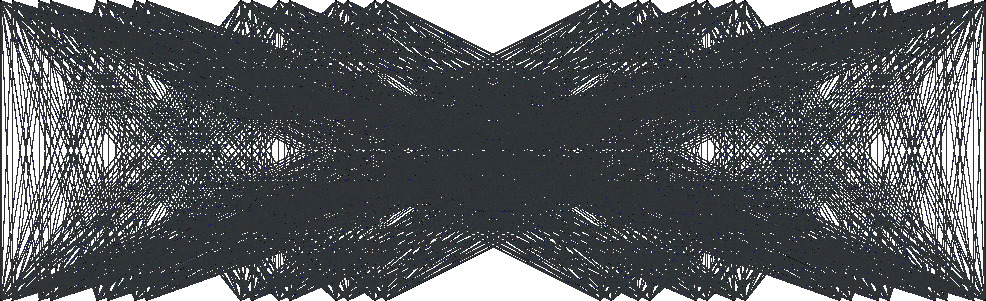} &
	$\leftarrow$
	& $\cdots$
\end{longtable}


\subsection{Solenoid} So far in our examples we only used mono-colored graphs. We shall use colored graphs to describe productions that describe the solenoid. We have two edge productions.

\vspace{3mm}
\begin{center}
\begin{tikzpicture}
\draw [densely dotted] (3,0)--(3,1.5);
\draw [densely dotted] (4,0)--(4.5,1.5);
\draw [->,shorten >= 3pt, shorten <= 3pt] (4,0.75)--(3,0.75);
\draw [->,shorten >= 3pt, shorten <= 3pt] (1,0.75)--(0,0.75);
\draw [line width=3pt,white] (4,1.5)--(4.5,0);
\draw (4,1.5)--(4.5,0);

\draw (0,0)--(0,1.5);
\draw (1,0)--(1,1.5);
\draw (1.5,0)--(1.5,1.5);
\foreach \t in {0,1,1.5, 3, 4, 4.5}{
\filldraw (\t,0) circle (\ptsize)
             (\t,1.5) circle (\ptsize);
}
\end{tikzpicture}
\end{center}

The inverse sequence corresponding to this diagram is

\vspace{1mm}
\begin{center}
\begin{tikzpicture}[scale=0.7]
\coordinate (A) at (0,0);
\coordinate (B) at (2,2);
\coordinate (C) at (2,-2);
\filldraw (A) circle(\ptsize);
\filldraw (B) circle(\ptsize);
\filldraw (C) circle(\ptsize);
\draw (B)--(A)--(C);
\draw [densely dotted] (B)--(C);
\coordinate (D) at (4,0);
\coordinate (E) at (6,2);
\coordinate (F) at (6,-2);
\coordinate (G) at (4.5,0);
\coordinate (H) at (6.5,2);
\coordinate (I) at (6.5,-2);
\draw [densely dotted] (H)--(F);
\draw [white, line width=3pt] (I)--(E);
\draw (E)--(I);
\draw [white, line width=3pt] (H)--(G);
\draw [white, line width=3pt] (I)--(G);

\filldraw (D) circle(\ptsize);
\filldraw (E) circle(\ptsize);
\filldraw (F) circle(\ptsize);
\filldraw (G) circle(\ptsize);
\filldraw (H) circle(\ptsize);
\filldraw (I) circle(\ptsize);

\draw (I)--(G)--(H);
\draw (E)--(D)--(F);

\coordinate (D') at (8,0);
\coordinate (E') at (10,2);
\coordinate (F') at (10,-2);
\coordinate (G') at (8.5,0);
\coordinate (H') at (10.5,2);
\coordinate (I') at (10.5,-2);

\coordinate (D'') at (8.15,0);
\coordinate (E'') at (10.15,2);
\coordinate (F'') at (10.15,-2);
\coordinate (G'') at (8.65,0);
\coordinate (H'') at (10.65,2);
\coordinate (I'') at (10.65,-2);

\draw [densely dotted] (F')--(H''); 
\draw [white, line width=2pt] (F'')--(H');
\draw (F'')--(H');
\draw [white, line width=2pt] (I')--(E');
\draw [white, line width=2pt] (I'')--(E'');
\draw (I')--(E');
\draw (I'')--(E'');
\draw [white, line width=2pt] (I')--(G');
\draw [white, line width=2pt] (I'')--(G'');
\draw (I')--(G');
\draw (I'')--(G'');
\draw [white, line width=2pt] (H')--(G');
\draw [white, line width=2pt] (H'')--(G'');
\draw (H')--(G');
\draw (H'')--(G'');
\draw [white, line width=2pt] (F'')--(D'');
\draw [white, line width=2pt] (D'')--(E'');
\draw (F')--(D')--(E');
\draw (F'')--(D'')--(E'');

\filldraw (D') circle(\ptsize);
\filldraw (E') circle(\ptsize);
\filldraw (F') circle(\ptsize);
\filldraw (G') circle(\ptsize);
\filldraw (H') circle(\ptsize);
\filldraw (I') circle(\ptsize);

\filldraw (D'') circle(\ptsize);
\filldraw (E'') circle(\ptsize);
\filldraw (F'') circle(\ptsize);
\filldraw (G'') circle(\ptsize);
\filldraw (H'') circle(\ptsize);
\filldraw (I'') circle(\ptsize);

\draw [->,shorten >= 10pt, shorten <= 10pt] (D)--(2,0);
\draw [->,shorten >= 10pt, shorten <= 10pt] (D')--(6.25,0);
\draw [->,shorten >= 10pt, shorten <= 10pt] (12.5,0)--(10.5,0);
\node at (12.7,0) {$\cdots$};

\end{tikzpicture}
\end{center}

The limit space is the solenoid~\cite{engelking1992}, directly from the definition.



\section{Topological properties}

\subsection{$k$-Connectedness and local $k$-connectedness} We prove the following sufficient condition for $k$-connectedness and local $k$-connectedness of a Markov space.

\begin{theorem}\label{thm:connectedness}
  Let $X$ be a Markov space that is the inverse limit of a sequence
  \[
     K_1\xleftarrow{p_1}K_2\xleftarrow{p_2}K_3\xleftarrow{p_3}\cdots .
  \]
  Assume that the sequence decomposes over a Markov diagram $\mathcal{D}$.
  If for each production $P$ in $\mathcal{D}$, $P$ is quasi-simplicial, the top complex $t(P)$ is $k$-connected for each $k < n$, and the starting space is $k$-connected for each $k < n$, then $X$ is $k$-connected and locally $k$-connected for each $k < n$.
\end{theorem}
\begin{proof}
  Fix $k < n$ and let $\varphi \colon S^k \to X$.   
  Let $\mathcal{D}_i$ be a decomposition of the space $K_i$ in the sequence. We will show that the inverse sequence $K_1 \xleftarrow{p_1} K_2 \xleftarrow{p_2} \cdots$ along with the sequence $\mathcal{D}_i$ of covers satisfies conditions of \cite[Lemma~4.3]{localk}. Then if $\varPhi \colon B^{k+1} \to K_1$ is a null-homotopy of $\pi_1 \circ \varphi$ (which exists because the starting space $K_1$ is $k$-connected), by \cite[Lemma~4.3]{localk} $\varPhi$ can be lifted to a map into $X$ that extends $\varphi$, which implies that $X$ is $k$-connected.
  
  To verify that conditions of \cite[Lemma~4.3]{localk} are satisfied we need to verify that the sequence $\mathcal{D}_i$ of covers satisfies conditions of \cite[Definition~4.2]{localk}.
  
  By Lemma~\ref{lem:metrics}, each $K_i$ is complete metric space (condition (A)) and each bonding map is $1$-Lipschitz (condition (B)). By the definition, each decomposition is closed, locally finite cover (condition (C)). By the assumption on top graphs, the pull-back $p_i^{-1}(\mathcal{D}_i)$ is $k$-connected and locally $k$-connected for each $k < n$. By \cite[Theorem 3.1]{localk}, it is an $AE(n)$-cover (condition (D)). Again, by Lemma~\ref{lem:metrics}, $\sum_{i=1}^\infty \mesh \mathcal{D}_i < \infty$ (condition (E)). We are done.
  
  To prove that $X$ is locally connected consider $x \in X$ and an open neighborhood $U$ of $x$. Let $i$ such that $B(\st_{\mathcal{D}_i} \pi_i(x), \sum_{k \geq i} \mesh \mathcal{D}_i) \subset U$. Let $V = \pi_i^{-1}(\st \pi_i(x))$. Observe that $\pi_i(V)$ is $k$-connected. If we repeat the above argument for a sequence $K_i \xleftarrow{p_i} K_{i+1} \xleftarrow{p_{i+1}}\cdots$, we get that any map $\varphi \colon S^k \to V$ is null-homotopic in $U$, hence $X$ is locally $k$-connected.
\end{proof}

\begin{remark}
  The combinatorial property that implies connectedness conditions stated in Theorem~\ref{thm:connectedness} is ``color-blind", i.e. it does not use colorings of the productions in any way.
\end{remark}

\begin{remark}
  The condition that production maps are quasi-simplicial in statement of Theorem~\ref{thm:connectedness} is necessary. In the suspension of the Cantor set example, the top graphs of the productions are connected, yet the limit is not locally connected. In the diamond sequence example the maps are quasi-simplicial and the diamond space is indeed locally connected.
\end{remark}

\subsection{Disjoint arcs property} We prove a sufficient condition for the disjoint arcs property of a compact Markov curve.

\begin{definition}
  A compact space $X$ has disjoint the arcs property if for each map $f \colon [0, 1] \times \{ 0, 1 \} \to X$ and
  each $\varepsilon > 0$ there exists a map $f' \colon [0, 1] \times \{ 0, 1 \} \to X$ that is $\epsilon$-close to $f$ and such that $f'([0,1] \times \{ 0 \}) \cap f'([0,1] \times \{ 1 \}) = \emptyset$.
\end{definition}
\begin{theorem}\label{thm:disjoint arcs}
  Let $X$ be a Markov space that is the inverse limit of a sequence
  \[
     K_1\xleftarrow{p_1}K_2\xleftarrow{p_2}K_3\xleftarrow{p_3}\cdots .
  \]
  Assume that for each $i$, $K_i$ is a finite graph.
  Assume that the sequence decomposes over an elementary Markov diagram $\mathcal{D}$ such that the vertex productions are isomorphic to
\begin{tikzpicture}
\coordinate (Y) at (0,0);
\fill (Y) circle(\ptsize);
\coordinate (E) at (1,0); \coordinate (D) at (2,0);
\fill (D) circle(\ptsize);
\fill (E) circle(\ptsize);
\draw [black] (D)--(E);
\draw [->,shorten >= 5pt, shorten <= 5pt] (E)--(Y);
\end{tikzpicture}
  and the edge productions have a connected and biconnected top graph. Then $X$ has the disjoint arcs property.
\end{theorem}
\begin{proof}
  As a first step we shall prove that for each $i$ there exist a pair of embeddings $f_i, g_i \colon K_i \to K_{i+1}$ such that $\im f_i \cap \im g_i = \emptyset$ and $p_i \circ f_i = p_i \circ g_i = \id_{K_i}$.

  For each vertex production, mark the two vertices of the top single-edge graph with the letters $A$ and $B$. Since the top graphs of the edge productions are biconnected, for each gluing of the vertex productions onto the end points of the edge production we may select disjoint paths in the top graph connecting vertices marked with $A$ and vertices marked with $B$. We construct $f_i$ to map the vertices of $K_i$ onto the corresponding vertices of $K_{i+1}$ marked with $A$ and such that it maps edges onto the selected paths. We perform a similar construction for $g_i$.
  
  Take $\varepsilon >0$. Fix $i$ such that the scale of $K_i$ is smaller than $\frac 13 \varepsilon$. Let $f \colon [0, 1] \times \{ 0, 1 \} \to X$. Observe that $f_i \circ \pi_i \circ f$ and $g_i \circ \pi_i \circ f$ are disjoint. By \cite[Lemma 4.3]{localk} we may lift these maps to $X$ in such a way that these lifts will be $\varepsilon$-close. Since their projections onto $K_i$ are disjoint, the maps are disjoint. We are done.
\end{proof}

\subsection{Application to the $1$--$8$ sequence}
We are now ready to prove that the inverse limit of the $1$--$8$ sequence defined in Subsection~\ref{ss:18} is homeomorphic to the Menger curve $M^3_1$.

\begin{theorem}
  The inverse limit of an elementary Markov sequence that decomposes over
production
\begin{center}
\begin{tikzpicture}[scale=0.75]
\coordinate (X) at (2,0); \coordinate (Y) at (0,0);
\draw [black] (X)--(Y); \fill (X) circle(\ptsize); \fill (Y) circle(\ptsize);
\coordinate (A) at (2,2); \coordinate (B) at (2,1);
\coordinate (F) at (1,2); \coordinate (C) at (1,1);
\coordinate (E) at (0,2); \coordinate (D) at (0,1);
\fill (A) circle(\ptsize); \fill (B) circle(\ptsize);
\fill (C) circle(\ptsize); \fill (D) circle(\ptsize);
\fill (E) circle(\ptsize); \fill (F) circle(\ptsize);
\draw [black] (A)--(B)--(C)--(D)--(E)--(F)--(A);
\draw [black] (F)--(C);
\draw [->,shorten >= 5pt, shorten <= 5pt] (C)--(1,0);
\end{tikzpicture}
\end{center}
with a connected finite starting graph is homeomorphic to the Menger curve $M^3_1$.
\end{theorem}
\begin{proof}
  Denote the limit space by $X$.
  Observe that the production maps of the Markov diagram are quasi-simplicial, have connected top graphs, and the starting graph is connected. By Theorem~\ref{thm:connectedness}, $X$ is connected and locally connected.
  Since each space in the sequence is a finite graph, the inverse limit is compact and at most $1$-dimensional. Since $X$ contains at least two points and is path-connected, it is at least $1$-dimensional.
  By Theorem~\ref{thm:disjoint arcs}, $X$ has the disjoint arcs property.

  We have shown that $X$ is compact, $1$-dimensional, connected, locally connected, and has disjoint arcs property.
  By Bestvina's characterization theorem~\cite{bestvina1988}, $X$ is homeomorphic to $M^3_1$.
\end{proof}








\bibliography{references}

\end{document}